\documentclass{amsart}

\usepackage{lmodern} 
\usepackage{microtype} 
\usepackage[english]{babel} 

\theoremstyle{plain} 
\newtheorem{theorem}{Theorem} 
\newtheorem{lemma}[theorem]{Lemma} 
\newtheorem{corollary}[theorem]{Corollary}

\DeclareMathOperator{\dist}{dist} 

\DeclareMathOperator{\Cov}{Cov}

\begin{document} 
\title{Hilbert points in Hilbert space-valued $L^p$ spaces} 
\date{\today} 

\author{Ole Fredrik Brevig} 
\address{Department of Mathematics, University of Oslo, 0851 Oslo, Norway} 
\email{obrevig@math.uio.no}

\author{Sigrid Grepstad} 
\address{Department of Mathematical Sciences, Norwegian University of Science and Technology (NTNU), NO-7491 Trondheim, Norway} 
\email{sigrid.grepstad@ntnu.no}

\begin{abstract}
	Let $H$ be a Hilbert space and $(\Omega,\mathcal{F},\mu)$ a probability space. A Hilbert point in $L^p(\Omega; H)$ is a nontrivial function $\varphi$ such that $\|\varphi\|_p \leq \|\varphi+f\|_p$ whenever $\langle f, \varphi \rangle = 0$. We demonstrate that $\varphi$ is a Hilbert point in $L^p(\Omega; H)$ for some $p\neq2$ if and only if $\|\varphi(\omega)\|_H$ assumes only the two values $0$ and $C>0$. We also obtain a geometric description of when a sum of independent Rademacher variables is a Hilbert point.
\end{abstract}


\thanks{Sigrid Grepstad is supported by Grant 275113 of the Research Council of Norway.} 

\maketitle

\section{Introduction} Suppose that $(\Omega,\mathcal{F},\mu)$ is a probability space and that $H$ is a Hilbert space. For $1 \leq p \leq \infty$, consider the usual $L^p$ spaces of $H$-valued Bochner integrable functions $f$ on $\Omega$. A \emph{Hilbert point} in $L^p(\Omega; H)$ is a nontrivial function $\varphi$ in $L^p(\Omega; H)$ which enjoys the property that 
\begin{equation}\label{eq:hilbertpoint} 
	\langle f, \varphi \rangle = 0 \qquad \implies \qquad \|\varphi\|_p \leq \|\varphi+f\|_p, 
\end{equation}
for every $f$ in $L^p(\Omega; H)$. The inner product in \eqref{eq:hilbertpoint} is that of the Hilbert space $L^2(\Omega; H)$, namely
\begin{equation}\label{eq:innerproduct}
	\langle f, g \rangle = \int_\Omega \langle f(\omega), g(\omega)\rangle_H \,d\mu(\omega).
\end{equation}
Every nontrivial function in $L^2(\Omega; H)$ is evidently a Hilbert point in $L^2(\Omega; H)$, since in this case $\langle f, \varphi \rangle = 0$ if and only if $\|\varphi+f\|_2^2 = \|\varphi\|_2^2 + \|f\|_2^2$.

If $p < 2$, then some care has to be taken when interpreting the inner product in \eqref{eq:hilbertpoint}. We declare here that $\langle f, \varphi \rangle = 0$ whenever $f$ lies in the $L^p(\Omega; H)$ closure of the set of functions $g$ in $L^\infty(\Omega; H)$ which satisfy $\langle g, \varphi \rangle = 0$. It turns out that this precaution is unnecessary, in view of our first main result.
\begin{theorem}\label{thm:general} 
	A nontrivial function $\varphi$ is a Hilbert point in $L^p(\Omega; H)$ for some $p\neq2$ if and only if there is a constant $C>0$ and a set $E$ with $\mu(E)>0$ such that 
	\begin{equation}\label{eq:general} 
		\|\varphi(\omega)\|_H = 
		\begin{cases}
			C, & \omega \in E, \\
			0, & \omega \not \in E. 
		\end{cases}
	\end{equation}
\end{theorem}

The functions satisfying \eqref{eq:general} are the eigenfunctions of the nonlinear operator associated with H\"older's inequality discussed in \cite[Section~1.2]{HSVZ18}. Equivalently, if $p^{-1}+q^{-1}=1$, then the nontrivial functions attaining equality in H\"older's inequality
\[\langle \varphi, \varphi \rangle \leq \|\varphi\|_p \|\varphi\|_q\]
are precisely those obeying \eqref{eq:general}. It is therefore hardly surprising that H\"older's inequality has a crucial role to play in the proof of Theorem~\ref{thm:general}.

The requirement \eqref{eq:general} does not depend on $p\neq2$, and so we observe the following corollary to Theorem~\ref{thm:general}: A nontrivial function $\varphi$ is a Hilbert point in $L^p(\Omega; H)$ for \emph{some} $p \neq 2$ if and only if it is a Hilbert point in $L^p(\Omega; H)$ for \emph{every} $1 \leq p \leq \infty$. 

The notion of Hilbert points was introduced in a recent paper of Brevig, Ortega-Cerd\`a and Seip \cite{BOCS21}. The focus of that paper was on the Hardy spaces of $d$-dimensional tori, denoted $H^p(\mathbb{T}^d)$. In contrast to the situation we have just observed, we recall from \cite[Section~6]{BOCS21} that in the context of $H^p(\mathbb{T}^d)$ a nontrivial function may be a Hilbert point for only \emph{one} exponent $p\neq2$. 

Another direct corollary of Theorem~\ref{thm:general} (with $\Omega = \mathbb{T}^d$ and $H=\mathbb{C}$) is the following: If a Hilbert point $\varphi$ in $H^p(\mathbb{T}^d)$ for some $p\neq2$ extends to a Hilbert point in the larger space $L^p(\mathbb{T}^d)$, then $\varphi$ is a constant multiple of an inner function. Theorem~\ref{thm:general} similarly provides a new proof of \cite[Corollary~2.5]{BOCS21}, which states that constant multiples of inner functions generate Hilbert points on $H^p(\mathbb{T}^d)$ for every $1 \leq p \leq \infty$.

Information about certain Hilbert points in $H^p(\mathbb{T}^d)$ was parlayed in \cite[Section~5]{BOCS21} to a new proof of the optimal constants in Khintchine's inequality for independent identically distributed Steinhaus variables for $2<p<\infty$. 

In the present paper, we will investigate independent identically distributed Rademacher variables. Let $(\omega_j)_{j\geq1}$ be a sequence of independent random variables taking the values $\pm 1$ with equal probability and consider the $H$-valued function 
\begin{equation}\label{eq:rademacher} 
	\varphi(\omega) = \sum_{j=1}^\infty \omega_j \mathbf{x}_j. 
\end{equation}
From the definition of the inner product \eqref{eq:innerproduct} and a computation, it follows that if $f$ denotes a function of the form \eqref{eq:rademacher}, then
\begin{equation} \label{eq:rademacher2norm}
	\|\varphi\|_{2}^2 = \sum_{j=1}^\infty \|\mathbf{x}_j\|_H^2.
\end{equation}
Our second main result is a characterization of the Hilbert points of the form \eqref{eq:rademacher}.
\begin{theorem}\label{thm:rademacher} 
	Let $\Omega=\{-1,1\}^\infty$ be the Cantor group and $\mu$ its Haar measure. The function \eqref{eq:rademacher} is a Hilbert point in $L^p(\Omega; H)$ for some $p \neq 2$ if and only if either 
	\begin{enumerate}
		\item[(a)] $(\mathbf{x}_j)_{j\geq1}$ is an orthogonal sequence satisfying $\sum_{j\geq1} \|\mathbf{x}_j\|_H^2 < \infty$. 
		\item[(b)] $\mathbf{x}_1=\mathbf{x}_2$ for a nonzero vector $\mathbf{x}_1$ and $\mathbf{x}_j=0$ for all $j \geq 3$. 
		\item[(c)] $\mathbf{u}$ and $\mathbf{v}$ are nonzero vectors satisfying $\|\mathbf{u}\|_H=\|\mathbf{v}\|_H$ and $\mathbf{u} \perp \mathbf{v}$, and
		\[\mathbf{x}_1 = \mathbf{u}, \qquad \mathbf{x}_2 = \frac{1}{2}\mathbf{u} + \frac{\sqrt{3}}{2} \mathbf{v}, \qquad \mathbf{x}_3 = \frac{1}{2}\mathbf{u} - \frac{\sqrt{3}}{2}\mathbf{v},\]
		and $\mathbf{x}_j = 0$ for all $j \geq 4$.
	\end{enumerate}
\end{theorem}

We close out this introduction with a probabilistic interpretation of Theorem~\ref{thm:rademacher}. Let us therefore think of $H$-valued functions $f$ on $\Omega$ as $H$-valued random variables. We refer to \cite[Section~6.1]{HNVW17} for the definitions of standard concepts used below. In particular, we note that the functions $\varphi$ in \eqref{eq:rademacher} are \emph{real symmetric}, since $\varphi$ and $-\varphi$ have the same distribution. It thus follows from \cite[Proposition~6.1.5]{HNVW17} that if $\varphi$ is of the form \eqref{eq:rademacher} and $f$ and $\varphi$ are \emph{independent}, then
\begin{equation} \label{eq:indep}
	\|f\|_p \leq \|f+\varphi\|_p
\end{equation}
for every $1 \leq p \leq \infty$. If $f$ is an integrable $H$-valued or scalar-valued random variable, then the expectation $\mathbb{E}(f) = \int_\Omega f(\omega)\,d\mu(\omega)$ exists. The \emph{covariance} of two $H$-valued random variables $f$ and $g$ is defined by 
\[\Cov(f,g) = \mathbb{E}(\langle f,g \rangle_H) - \langle \mathbb{E}(f), \mathbb{E}(g) \rangle_{H},\]
and two $H$-valued random variables are said to be \emph{uncorrelated} if $\Cov(f,g)=0$. Since the real-symmetric variables $\varphi$ from \eqref{eq:rademacher} satisfy $\mathbb{E}(\varphi)=\mathbf{0}$, we have the following reformulation of Theorem~\ref{thm:rademacher} which should be compared with \eqref{eq:indep}: Fix $p\neq2$. The $H$-valued random variables of the form \eqref{eq:rademacher} satisfying the requirement $\|\varphi\|_p \leq \|f+\varphi\|_p$ for every $f$ which is uncorrelated to $\varphi$ are precisely those given by Theorem~\ref{thm:rademacher}.

\subsection*{Organization} This paper is comprised of two further sections. Section~\ref{sec:general} is devoted to the proof of Theorem~\ref{thm:general}. The proof of Theorem~\ref{thm:rademacher} can be found in Section~\ref{sec:geometry}.

\section{Proof of Theorem~\ref{thm:general}} \label{sec:general} 
We split the proof of Theorem~\ref{thm:general} into three parts, and present these in order of increasing difficulty.
\begin{proof}
	[Proof of Theorem~\ref{thm:general}: Sufficiency] Assume that \eqref{eq:general} holds, meaning that there is a constant $C>0$ and a set $E$ with $\mu(E)>0$ such that
	\[\|\varphi(\omega)\|_H = 
	\begin{cases}
		C, & \omega \in E, \\
		0, & \omega \not \in E. 
	\end{cases}
	\]
	Our goal is to show that $\varphi$ is a Hilbert point in $L^p(\Omega; H)$. Since $\varphi$ is bounded, it is clear that $\varphi$ is in $L^p(\Omega; H)$ for every $1 \leq p \leq \infty$. In particular, the orthogonal projection 
	\begin{equation}\label{eq:orthogonalproj} 
		P_\varphi f = \frac{\langle f, \varphi \rangle}{\|\varphi\|_2^2} \varphi 
	\end{equation}
	extends to a bounded operator on $L^p(\Omega; H)$. Using H\"older's inequality, we find that
	\[\|P_\varphi\|_{L^p(\Omega; H) \to L^p(\Omega; H)} \leq \frac{\|\varphi\|_q}{\|\varphi\|_2^2} \|\varphi\|_p = 1.\]
	The final equality follows from the fact that $\varphi$ satisfies \eqref{eq:general}. Now if $\langle f, \varphi \rangle = 0$, then
	\[\|\varphi\|_p = \|P_\varphi(\varphi + f)\|_p \leq \|\varphi +f\|_p ,\]
	which demonstrates that $\varphi$ is a Hilbert point in $L^p(\Omega; H)$. 
\end{proof}

The necessity part in the proof of Theorem~\ref{thm:general} requires a separate argument for the case $p=\infty$.

\begin{proof}
	[Proof of Theorem~\ref{thm:general}: Necessity for $p=\infty$] Since $\varphi$ is in $L^\infty(\Omega; H)$ the orthogonal projection $P_\varphi$ in \eqref{eq:orthogonalproj} is a bounded operator on $L^\infty(\Omega, H)$. Every $f$ in $L^\infty(\Omega; H)$ can be orthogonally decomposed as $f = c \varphi + g$, so the assumption that $\varphi$ is a Hilbert point in $L^\infty(\Omega; H)$ implies that
	\[\|P_\varphi f \|_\infty = \|c\varphi\|_\infty \leq \|c\varphi + g\|_\infty = \|f\|_\infty.\]
	Here we tacitly used the fact that $\varphi$ is Hilbert point if and only if $c\varphi$ is a Hilbert point for every constant $c\neq0$. This shows that $\|P_\varphi\|_{L^\infty(\Omega, H) \to L^\infty(\Omega; H)} \leq 1$. Consider
	\[\psi(\omega) = 
	\begin{cases}
		\frac{\varphi(\omega)}{\|\varphi(\omega)\|_H}, & \text{if } \varphi(\omega) \neq \mathbf{0}, \\
		\mathbf{0}, & \text{if } \varphi(\omega) = \mathbf{0}. 
	\end{cases}
	\]
	Since $\varphi$ is nontrivial, it is clear that $\|\psi\|_{L^\infty(\Omega,H)}=1$. Hence $\|P_\varphi \psi\|_\infty \leq 1$, which means that
	\[\frac{\|\varphi\|_1}{\|\varphi\|_2^2} \|\varphi\|_{\infty} \leq 1 \qquad \iff \qquad \|\varphi\|_1 \|\varphi\|_\infty \leq \langle \varphi, \varphi \rangle.\]
	The final inequality is actually an equality, since the reverse inequality is H\"older's inequality. This is only possible if \eqref{eq:general} holds. 
\end{proof}

The proof we have just given can easily be adapted to work also for $2<p<\infty$. However, we run into problems for $1 \leq p < 2$ since we cannot guarantee that the orthogonal projection $P_\varphi$ in \eqref{eq:orthogonalproj} is well-defined. The issue is simply that we cannot guarantee a priori that $\varphi$ is in $L^2(\Omega; H)$. We circumvent this problem by using the Riesz representation theorem for $L^p(\Omega; H)$, which holds in our context since every Hilbert space enjoys the Radon--Nikodym property. We refer broadly to the monograph of Diestel and Uhl \cite[Chapter~IV]{DU77}. The following proof is inspired by arguments from \cite[Section~4.2]{Shapiro71}.
\begin{proof}
	[Proof of Theorem~\ref{thm:general}: Necessity for $p<\infty$] Suppose that $\varphi$ is a Hilbert point in $L^p(\Omega; H)$ where $1 \leq p < 2$ or $2<p<\infty$. Our first goal is to show that we may assume without loss of generality that $\varphi$ does not vanish\footnote{This reduction is technically only needed in the proof of the case $p=1$, but it simplifies the overall exposition also for $p>1$.}. Set $N = \varphi^{-1}(\{\mathbf{0}\})$. The behavior of $f$ on $N$ does not affect the inner product $\langle f, \varphi \rangle$ and may only increase the norm $\|\varphi+f\|_p$. When checking the Hilbert point condition, it is therefore sufficient to consider only $f$ that vanish on $N$. Replacing $\Omega$ with $\Omega \setminus N$, we may therefore assume that $\varphi$ does not vanish.
	
	Let $X$ be the subspace of $L^p(\Omega; H)$ formed by taking the closure of the set of bounded functions $g$ which satisfy $\langle g, \varphi \rangle = 0$. The assumption that $\varphi$ is a Hilbert point in $L^p(\Omega; H)$ means that if $f$ is in $X$, then $\|\varphi\|_p \leq \|\varphi+f\|_p$. This implies that 
	\begin{equation}\label{eq:dist} 
		\|\varphi\|_p = \dist(\varphi,X) = \inf_{f \in X} \|\varphi+f\|_p. 
	\end{equation}
	It follows from \eqref{eq:dist} and the Hahn--Banach theorem that there exists a linear functional $\Phi$ on $L^p(\Omega; H)$ such that $\Phi(\varphi)=1$, $\Phi(f)=0$ for every $f$ in $X$ and $\|\Phi\| = \|\varphi\|_p^{-1}$. In particular, the norm of $\Phi$ is attained at $\varphi$. Now we bring into play the Riesz representation theorem, which tells us that $\Phi(\cdot) = \langle \cdot, \psi \rangle$ for some unique function $\psi$ in $L^q(\Omega; H)$ where $p^{-1}+q^{-1}=1$. However, this means that
	\[\frac{1}{\|\varphi\|_p} = \|\Phi\| = \|\psi\|_q \qquad \text{and} \qquad 1 = \Phi(\varphi) = \langle \varphi, \psi \rangle,\]
	which shows that $\langle \varphi, \psi \rangle = \|\varphi\|_p \|\psi\|_q$. Since we have attained equality in H\"older's inequality and since $\varphi$ does not vanish, we necessarily conclude that
	\[\psi(\omega) = \frac{\|\varphi(\omega)\|_H^{p-2}}{\|\varphi\|_p^p} \varphi(\omega).\]
	
	In view of what we have just established, it follows that if $g$ is in $L^\infty(\Omega; H)$, then 
	\begin{equation}\label{eq:XiffY} 
		\langle g, \varphi \rangle = 0 \qquad \implies \qquad \langle g, \|\varphi\|_H^{p-2} \varphi \rangle = 0. 
	\end{equation}
	Let $F_1$ and $F_2$ be sets of positive measure and define
	\[g = \left(\int_{F_1} \|\varphi\|_H \,d\mu \right)^{-1} \chi_{F_1} \frac{\varphi}{\|\varphi\|_H} - \left(\int_{F_2} \|\varphi\|_H \,d\mu \right)^{-1} \chi_{F_2} \frac{\varphi}{\|\varphi\|_H}.\]
	Here $\chi_{F_1}$ and $\chi_{F_2}$ are the characteristic (scalar) functions of $F_1$ and $F_2$, respectively. Since $\varphi$ does not vanish and since $\varphi$ is in $L^1(\Omega; H)$ by H\"older's inequality, it is clear that $g$ is in $L^\infty(\Omega; H)$ and that $\langle g, \varphi \rangle = 0$. It therefore follows from \eqref{eq:XiffY} that 
	\begin{equation}\label{eq:EF} 
		\langle g, \|\varphi\|_H^{p-2} \varphi \rangle = 0 \qquad \implies \qquad \frac{\int_{F_1} \|\varphi\|_H^{p-1}\,d\mu }{ \int_{F_1} \|\varphi\|_H\,d\mu} = \frac{\int_{F_2} \|\varphi\|_H^{p-1}\,d\mu }{ \int_{F_2} \|\varphi\|_H\,d\mu}. 
	\end{equation}
	There are now two cases. If $p>2$, then it follows from \eqref{eq:EF} that 
	\begin{equation}\label{eq:infsup} 
		\inf_{\omega \in F_1} \|\varphi(\omega)\|_H^{p-2} \leq \frac{\int_{F_1} \|\varphi\|_H^{p-1}\,d\mu }{ \int_{F_1} \|\varphi\|_H\,d\mu} = \frac{\int_{F_2} \|\varphi\|_H^{p-1}\,d\mu }{ \int_{F_2} \|\varphi\|_H\,d\mu} \leq \sup_{\omega \in F_2} \|\varphi(\omega)\|_H^{p-2}. 
	\end{equation}
	Since $F_1$ and $F_2$ are arbitrary sets of positive measure, we deduce from this that $\|\varphi(\omega)\|_H^{p-2} = C$ for almost every $\omega$ in $\Omega$ which means that \eqref{eq:general} holds. The same argument works for $1 \leq p < 2$, provided we first swap $\inf$ and $\sup$ in \eqref{eq:infsup}. 
\end{proof}

As discussed in \cite{BOCS21}, the term Hilbert point is meant to indicate that we are dealing with points in a Banach space around which the geometry of the Banach space behaves like a Hilbert space. In the present paper, we restrict ourselves to mentioning the following immediate corollary of Theorem~\ref{thm:general}.

\begin{corollary} \label{cor:orthogonalporjection}
	A nontrivial function $\varphi$ is a Hilbert point in $L^p(\Omega; H)$ for some $p\neq2$ if and only if $\varphi$ is in $L^2(\Omega; H)$ and the orthogonal projection
	\[P_\varphi f = \frac{\langle f, \varphi \rangle }{\|\varphi\|_2^2} \varphi\]
	extends to a norm $1$ operator on $L^p(\Omega; H)$.
\end{corollary}

\begin{proof}
	If $\varphi$ is a Hilbert point in $L^p(\Omega; H)$, then it follows from Theorem~\ref{thm:general} that $\varphi$ is in $L^2(\Omega; H)$ and that $\|P_\varphi\|_{L^p(\Omega; H) \to L^p(\Omega; H)} \leq 1$ by H\"older's inequality. To see that actually $\|P_\varphi\|_{L^p(\Omega; H) \to L^p(\Omega; H)} = 1$ it is sufficient to set $f = \varphi$.
	
	Conversely, suppose that $\varphi$ is in $L^2(\Omega; H)$ and that the orthogonal projection $P_\varphi$ extends to a norm $1$ operator on $L^p(\Omega; H)$. It now follows at once that $\varphi$ is a Hilbert point since if $\langle f, \varphi \rangle =0$, then $\|\varphi\|_p = \|P_\varphi(\varphi+f)\|_p \leq \|\varphi+f\|_p$. 
\end{proof}

\section{Proof of Theorem~\ref{thm:rademacher}} \label{sec:geometry} 
Let it be known that all norms and inner products in the present section will be with respect to some fixed Hilbert space $H$. This is a notational departure from what has been previously employed.

To warm up, we use Theorem~\ref{thm:general} to check that the three cases (a), (b) and (c) of Theorem~\ref{thm:rademacher} indeed describe Hilbert points $\varphi$.
\begin{proof}
	[Proof of Theorem~\ref{thm:rademacher}: Sufficiency] We are going to use Theorem~\ref{thm:general}. 
	\begin{enumerate}
		\item[(a)] If the sequence $(\mathbf{x}_j)_{j\geq1}$ is orthogonal and $\sum_{j\geq1} \|\mathbf{x}_j\|^2 < \infty$, then
		\[\|\varphi(\omega)\|^2 = \sum_{j=1}^{\infty} \|\mathbf{x}_j\|^2 = C\]
		for every $\omega$ in $\Omega$ since $|\pm1|=1$. Hence $\varphi$ is a Hilbert point. 
		\item[(b)] If $\varphi(\omega)=\omega_1 \mathbf{x}_1+ \omega_2 \mathbf{x}_1$, then $\varphi(\pm1,\pm1) = \pm2\mathbf{x}_1$ and $\varphi(\pm1,\mp1) = \mathbf{0}$. Since either $\|\varphi(\omega)\|=0$ or $\|\varphi(\omega)\|=2\|\mathbf{x}_1\|=C$, we see that $\varphi$ is a Hilbert point. 
		\item[(c)] Suppose that $\mathbf{u}$ and $\mathbf{v}$ are vectors of equal length with $\mathbf{u} \perp \mathbf{v}$ and set
		\[\varphi(\omega) = \omega_1 \mathbf{u} + \omega_2 \left(\frac{1}{2}\mathbf{u}+\frac{\sqrt{3}}{2}\mathbf{v}\right)+\omega_3 \left(\frac{1}{2}\mathbf{u}-\frac{\sqrt{3}}{2}\mathbf{v}\right).\]
		There are eight choices of $\omega$. We first compute 
		\begin{align*}
			\varphi(1,1,1) &= 2\mathbf{u}, \\
			\varphi(-1,1,1) &= \mathbf{0}, \\
			\varphi(1,\pm1,\mp1) &= \mathbf{u} \pm \sqrt{3} \mathbf{v}. 
		\end{align*}
		The four remaining choices of $\omega$ produce the same vectors multiplied by $-1$. For our purposes it is therefore sufficient to consider these four. Since $\mathbf{u} \perp \mathbf{v}$ and $\|\mathbf{u}\|=\|\mathbf{v}\|$ we get
		\[\|\mathbf{u} \pm \sqrt{3} \mathbf{v}\| = \sqrt{\|\mathbf{u}\|^2 + 3 \|\mathbf{v}\|^2} = 2 \|\mathbf{u}\|=C,\]
		and hence $\varphi$ is a Hilbert point. \qedhere 
	\end{enumerate}
\end{proof}

The more difficult ``necessity'' part of the proof requires some preparation in the form of a trio of lemmas on elementary Hilbert space geometry. 
\begin{lemma}\label{lem:geom1} 
	Suppose that $\mathbf{u}_0,\mathbf{u}_1,\mathbf{u}_2$ are nonzero vectors in $H$ such that
	\[\|\mathbf{u}_0\|=\|\mathbf{u}_0+\mathbf{u}_1\|=\|\mathbf{u}_0+\mathbf{u}_2\|.\]
	\begin{enumerate}
		\item[(a)] If $\|\mathbf{u}_0+\mathbf{u}_1+\mathbf{u}_2\|=0$, then
		\[\mathbf{u}_1 = -\frac{1}{2}\mathbf{u}_0 + \frac{\sqrt{3}}{2}\mathbf{v} \qquad \text{and} \qquad \mathbf{u}_2 = -\frac{1}{2}\mathbf{u}_0 - \frac{\sqrt{3}}{2}\mathbf{v},\]
		where $\mathbf{v}$ is a vector which satisfies $\mathbf{v} \perp \mathbf{u}_0$ and $\|\mathbf{v}\|=\|\mathbf{u}_0\|$. 
		\item[(b)] If $\|\mathbf{u}_0+\mathbf{u}_1+\mathbf{u}_2\|=\|\mathbf{u}_0\|$, then $\mathbf{u}_1 \perp \mathbf{u}_2$. 
	\end{enumerate}
\end{lemma}
\begin{proof}
	For (a) we see that $\|\mathbf{u}_0 + \mathbf{u}_1 + \mathbf{u}_2\|=0$ implies that $\mathbf{u}_1 = -(\mathbf{u}_0+\mathbf{u}_2)$. Since $\|\mathbf{u}_0+\mathbf{u}_2\|=\|\mathbf{u}_0+\mathbf{u}_1\|$ by assumption, we see that $\|\mathbf{u}_1\|=\|\mathbf{u}_0+\mathbf{u}_1\|$. Expanding, we obtain
	\[\|\mathbf{u}_1\|^2 = \|\mathbf{u}_0\|^2 + 2 \langle \mathbf{u}_0,\mathbf{u}_1 \rangle + \|\mathbf{u}_1\|^2,\]
	which means that $\mathbf{u}_1 = -\frac{1}{2}\mathbf{u}_0 + \widetilde{\mathbf{u}}_1$ where $\widetilde{\mathbf{u}}_1 \perp \mathbf{u}_0$. It follows that
	\[\mathbf{u}_2 = -(\mathbf{u}_0+\mathbf{u}_1)= -\frac{1}{2}\mathbf{u}_0-\widetilde{\mathbf{u}}_1,\]
	and since $\|\mathbf{u}_1\|=\|\mathbf{u}_0\|$ we conclude that $\|\widetilde{\mathbf{u}}_1\|=\frac{\sqrt{3}}{2} \|\mathbf{u}_0\|$. This implies the stated result.
	
	For (b) we compute 
	\begin{align*}
		\|\mathbf{u}_0+\mathbf{u}_1+\mathbf{u}_2\|^2 &= \|\mathbf{u}_0+\mathbf{u}_1\|^2 + 2 \langle \mathbf{u}_0+\mathbf{u}_1,\mathbf{u}_2 \rangle + \|\mathbf{u}_2\|^2 \\
		&= \|\mathbf{u}_0+\mathbf{u}_1\|^2 + 2\langle \mathbf{u}_1, \mathbf{u}_2 \rangle + \|\mathbf{u}_0+\mathbf{u}_2\|^2 - \|\mathbf{u}_0\|^2. 
	\end{align*}
	Since $\|\mathbf{u}_0+\mathbf{u}_1+\mathbf{u}_2\|=\|\mathbf{u}_0+\mathbf{u}_1\|=\|\mathbf{u}_0+\mathbf{u}_2\|=\|\mathbf{u}_0\|$, we see that $\langle \mathbf{u}_1, \mathbf{u}_2 \rangle =0$. 
\end{proof}
\begin{lemma}\label{lem:geom2} 
	Suppose that $\mathbf{u}_0,\mathbf{u}_1,\mathbf{u}_2,\mathbf{u}_3$ are nonzero vectors in $H$ such that
	\[\|\mathbf{u}_0\| = \|\mathbf{u}_0+\mathbf{u}_1\| = \|\mathbf{u}_0+\mathbf{u}_2\| = \|\mathbf{u}_0+\mathbf{u}_3\|\]
	and that
	\[\|\mathbf{u}_0+\mathbf{u}_1+\mathbf{u}_2\|,\quad \|\mathbf{u}_0+\mathbf{u}_1+\mathbf{u}_3\|,\quad \|\mathbf{u}_0+\mathbf{u}_2+\mathbf{u}_3\|\quad \text{and}\quad \|\mathbf{u}_0+\mathbf{u}_1+\mathbf{u}_2+\mathbf{u}_3\|\]
	are (independently) equal either to $0$ or to $\|\mathbf{u}_0\|$. Then
	\[\|\mathbf{u}_0+\mathbf{u}_1+\mathbf{u}_2\|=\|\mathbf{u}_0+\mathbf{u}_1+\mathbf{u}_3\|=\|\mathbf{u}_0+\mathbf{u}_2+\mathbf{u}_3\|=\|\mathbf{u}_0+\mathbf{u}_1+\mathbf{u}_2+\mathbf{u}_3\|=\|\mathbf{u}_0\|.\]
\end{lemma}
\begin{proof}
	We begin by demonstrating that 
	\begin{equation}\label{eq:contradict1} 
		\|\mathbf{u}_0+\mathbf{u}_1+\mathbf{u}_2+\mathbf{u}_3\|=0 
	\end{equation}
	is impossible. If \eqref{eq:contradict1} holds, then $\mathbf{u}_1 = -(\mathbf{u}_0+\mathbf{u}_2+\mathbf{u}_3)$. Hence $\|\mathbf{u}_1\| = \|\mathbf{u}_0+\mathbf{u}_2+\mathbf{u}_3\|$ is equal to either $0$ or $\|\mathbf{u}_0\|$. By our assumption that $\mathbf{u}_1$ is nonzero it follows that $\|\mathbf{u}_1\|=\|\mathbf{u}_0\|$. Combining this with the assumption that $\|\mathbf{u}_0+\mathbf{u}_1\|=\|\mathbf{u}_0\|$ as in the proof of Lemma~\ref{lem:geom1}~(a), we conclude that $\mathbf{u}_1 = -\frac{1}{2}\mathbf{u}_0+\widetilde{\mathbf{u}}_1$ where $\widetilde{\mathbf{u}}_1\perp \mathbf{u}_0$. By symmetry, the same holds also for $\mathbf{u}_2$ and $\mathbf{u}_3$. By orthogonality, we find that
	\[\|\mathbf{u}_0+\mathbf{u}_1+\mathbf{u}_2+\mathbf{u}_3\| \geq \frac{1}{2} \|\mathbf{u}_0\|,\]
	contradicting the assumption that $\mathbf{u}_0$ is nonzero in view of \eqref{eq:contradict1}. Hence 
	\begin{equation}\label{eq:correct1} 
		\|\mathbf{u}_0+\mathbf{u}_1+\mathbf{u}_2+\mathbf{u}_3\|=\|\mathbf{u}_0\|. 
	\end{equation}
	By symmetry, it remains to show that 
	\begin{equation}\label{eq:contradict2} 
		\|\mathbf{u}_0+\mathbf{u}_1+\mathbf{u}_2\|=0 
	\end{equation}
	is impossible. If \eqref{eq:contradict2} holds, then $\mathbf{u}_1 = -(\mathbf{u}_0+\mathbf{u}_2)$. As in the proof of Lemma~\ref{lem:geom1}~(a) we get that $\mathbf{u}_1 = -\frac{1}{2}\mathbf{u}_0 + \widetilde{\mathbf{u}}_1$ where $\widetilde{\mathbf{u}}_1 \perp \mathbf{u}_0$ and $\|\widetilde{\mathbf{u}}_1\|=\frac{\sqrt{3}}{2} \|\mathbf{u}_0\|$. By \eqref{eq:contradict2} we find that $\mathbf{u}_2 = -\frac{1}{2}\mathbf{u}_0 - \widetilde{\mathbf{u}}_1$. By \eqref{eq:correct1}, we know that $\|\mathbf{u}_3\|=\|\mathbf{u}_0\|$, which means that $\mathbf{u}_3 = - \frac{1}{2}\mathbf{u}_0+\widetilde{\mathbf{u}}_3$ where $\widetilde{\mathbf{u}}_3\perp \mathbf{u}_0$ and $\|\widetilde{\mathbf{u}}_3\|=\frac{\sqrt{3}}{2}\|\mathbf{u}_0\|$. We now look at the final two expressions
	\[\|\mathbf{u}_0+\mathbf{u}_1+\mathbf{u}_3\| = \|\widetilde{\mathbf{u}}_1+\widetilde{\mathbf{u}}_3\| \qquad \text{and} \qquad \|\mathbf{u}_0+\mathbf{u}_2+\mathbf{u}_3\| = \|-\widetilde{\mathbf{u}}_1+\widetilde{\mathbf{u}}_3\|,\]
	which are by assumption either equal to $0$ or to $\|\mathbf{u}_0\|$. There are two possible cases. 
	\begin{enumerate}
		\item[1.] If (at least) one is equal to $0$, say $\|-\widetilde{\mathbf{u}}_1+\widetilde{\mathbf{u}}_3\| =0$, then $\widetilde{\mathbf{u}}_3=\widetilde{\mathbf{u}}_1$ and hence
		\[\|\widetilde{\mathbf{u}}_1+\widetilde{\mathbf{u}}_3\| = 2\|\widetilde{\mathbf{u}}_1\|=\sqrt{3}\|\mathbf{u}_0\| .\]
		This contradicts the assumption that $\|\widetilde{\mathbf{u}}_1+\widetilde{\mathbf{u}}_3\|$ is equal to $0$ or $\|\mathbf{u}_0\|$, since $\mathbf{u}_0$ is nonzero by assumption. A similar contradiction is reached if we start by assuming that $\| \widetilde{\mathbf{u}}_1+ \widetilde{\mathbf{u}}_3\|=0$. 
		\item[2.] If both are equal to $\|\mathbf{u}_0\|$, then
		\[0 = \|\widetilde{\mathbf{u}}_1+\widetilde{\mathbf{u}}_3\|^2 - \|-\widetilde{\mathbf{u}}_1+\widetilde{\mathbf{u}}_3\|^2 = 4 \langle \widetilde{\mathbf{u}}_1,\widetilde{\mathbf{u}}_3 \rangle.\]
		This shows that $\|\widetilde{\mathbf{u}}_1+\widetilde{\mathbf{u}}_3\|^2 = \frac{3}{2}\|\mathbf{u}_0\|^2$ which cannot be equal to $\|\mathbf{u}_0\|^2$ since $\mathbf{u}_0$ is nonzero by assumption. 
	\end{enumerate}
	We conclude that \eqref{eq:contradict2} is false and hence $\|\mathbf{u}_0+\mathbf{u}_1+\mathbf{u}_2\|=\|\mathbf{u}_0\|$. 
\end{proof}

For a sequence of vectors $(\mathbf{u}_j)_{j \geq 1}$ in $H$ and a subset $J$ of $\mathbb{N}$, let
\[\mathbf{u}(J) = \sum_{j \in J} \mathbf{u}_j,\]
and let $|J|$ denote the cardinality of $J$. We are now ready to establish the third and final geometric lemma. 
\begin{lemma}\label{lem:geom3} 
	Let $\mathbf{u}_0$ be a nonzero vector in $H$, and for an index set $J$ with $|J|\geq 3$ suppose that $(\mathbf{u}_j)_{j \in J}$ is a sequence of nonzero vectors in $H$ with
	\[\|\mathbf{u}_0+\mathbf{u}_j\| = \|\mathbf{u}_0\| \]
	for all $j$ in $J$. If $\|\mathbf{u}_0+\mathbf{u}(K)\|$ is (independently) equal to either $0$ or $\|\mathbf{u}_0\|$ for every finite subset $K$ of $J$, then $\|\mathbf{u}_0+\mathbf{u}(K)\|=\|\mathbf{u}_0\|$ for every finite subset $K$ of $J$. 
\end{lemma}
\begin{proof}
	The proof is based on Lemma~\ref{lem:geom2} and induction. If $|K|=0$ or $|K|=1$, then there is nothing to prove. If $|K|=2$ or $|K|=3$, then the result follows directly from Lemma~\ref{lem:geom2}. Suppose therefore that $|K|\geq4$. Without loss of generality, we may assume that $\{1,2,3\} \subset K$ and write
	\[\mathbf{u}(K) = \mathbf{u}_1 + \mathbf{u}_2 + \mathbf{u}_3 + \mathbf{u}(\widetilde{K}),\]
	where $\widetilde{K} = K \setminus \{1,2,3\}$. By the induction hypothesis,
	\[\|\mathbf{u}_0+\mathbf{u}(\widetilde{K})\| = \|\mathbf{u}_0+\mathbf{u}(\widetilde{K})+\mathbf{u}_1\| = \|\mathbf{u}_0+\mathbf{u}(\widetilde{K})+\mathbf{u}_2\| = \|\mathbf{u}_0+\mathbf{u}(\widetilde{K})+\mathbf{u}_3\| = \|\mathbf{u}_0\|.\]
	By Lemma~\ref{lem:geom2} (with $\mathbf{u}_0$ replaced by $\mathbf{u}_0 + \mathbf{u}(\widetilde{K})$), we find that $\|\mathbf{u}_0+\mathbf{u}(K)\|=\|\mathbf{u}_0\|$. 
\end{proof}
\begin{proof}
	[Proof of Theorem~\ref{thm:rademacher}: Necessity] Suppose that $\varphi$ is a Hilbert point in $L^p(\Omega; H)$ of the form
	\[\varphi(\omega) = \sum_{j=1}^\infty \omega_j \mathbf{x}_j.\]
	Since $\varphi$ is in $L^p(\Omega; H)$ by assumption, we can combine \eqref{eq:rademacher2norm} with Khintchine's inequality (see e.g.~\cite[Corollary~3.2.24]{HNVW16}) for $1 \leq p<2$ or with H\"older's inequality for $2<p \leq \infty$, to conclude that
	\[\sum_{j=1}^\infty \|\mathbf{x}_j\|^2 < \infty.\]
	It now follows from \cite[Theorem~3.2]{Kahane85} that the series $\varphi(\omega)$ converges in $H$ for almost every $\omega$ in $\Omega$. We may replace $\mathbf{x}_j$ by $-\mathbf{x}_j$ without affecting whether $\varphi$ is a Hilbert point, which allows us to assume without loss of generality that
	\[\varphi(\mathbf{1}) = \sum_{j=1}^\infty \mathbf{x}_j\]
	converges in $H$. If $\varphi$ is not identically equal to $\mathbf{0}$, we may also assume that $\varphi(\mathbf{1}) \neq \mathbf{0}$. Let us for simplicity set $\mathbf{u}_0 = \varphi(\mathbf{1})$.
	
	If we start from $\omega = \mathbf{1}$ and change the sign of $\omega_j$ for a finite set of indices, then the series for $\varphi(\omega)$ will remain convergent in $H$. Specifically, changing the sign of one $\omega_j$ we obtain the vectors 
	\begin{equation}\label{eq:flip} 
		\varphi(\omega) = \mathbf{u}_0 - 2 \mathbf{x}_j 
	\end{equation}
	for $j=1,2,3,\ldots$. Since $\varphi$ is a Hilbert point by assumption, we know from Theorem~\ref{thm:general} that the norms of each of the vectors in \eqref{eq:flip} are equal to either $0$ or to $\|\mathbf{u}_0\|$. Let $J$ denote the index set of those vectors in $(\mathbf{x}_j)_{j \geq 1}$ which are nonzero and which satisfy $\|\mathbf{u}_0-2\mathbf{x}_j\|=\|\mathbf{u}_0\|$. There are four cases to be considered. 
	\begin{enumerate}
		\item[1.] If $|J|=0$, then any nonzero vector $\mathbf{x}_j$ satisfies $\|\mathbf{u}_0-2\mathbf{x}_j\|=0$, and thus $\mathbf{x}_j = \frac{1}{2}\mathbf{u}_0$. Since $\sum_{j\geq 1} \mathbf{x}_j=\mathbf{u}_0$, it follows that $\mathbf{x}_j =0$ for all but two indices, and we are in case (b) of Theorem~\ref{thm:rademacher}. 
		
		\item[2.] If $|J|=1$, then we can without loss of generality assume that $J=\{1\}$. For all $j\geq 2$ we therefore have either $\mathbf{x}_j=0$ or $\|\mathbf{u}_0-2\mathbf{x}_j \|=0$. In the latter case it follows that $\mathbf{x}_j = \frac{1}{2}\mathbf{u}_0$. Hence
		\[\mathbf{u}_0-\mathbf{x}_1 = \sum_{j=2}^{\infty} \mathbf{x}_j = c\mathbf{u}_0,\]
		where $c=k/2$ for a nonnegative integer $k$. This is only possible if $\mathbf{x}_j=0$ for all but finitely many $j$. We get $\mathbf{x}_1=(1-c)\mathbf{u}_0$, and since $\|\mathbf{u}_0-2\mathbf{x}_1\|=\|\mathbf{u}_0\|$ and $\mathbf{x}_1\neq 0$, it follows that $c=0$. This means that we are in case (a) of Theorem~\ref{thm:rademacher} with only one nonzero vector in the sequence.
		
		\item[3.] In the case $|J|=2$, we use Lemma~\ref{lem:geom1}. We assume without loss of generality that $J=\{1, 2\}$, and consider two subcases. 
		
		First, if
		\[\|\mathbf{u}_0-2(\mathbf{x}_1+\mathbf{x}_2)\|=\|\mathbf{u}_0\|,\]
		then by Lemma~\ref{lem:geom1}~(b) we conclude that $\mathbf{x}_1 \perp \mathbf{x}_2$. Suppose for the purposes of contradiction that there is some $\mathbf{x}_k$ with $k\geq3$ such that $\|\mathbf{u}_0-2\mathbf{x}_k\| =0$ and $\mathbf{x}_k=\frac{1}{2}\mathbf{u}_0$. Choosing $\omega_j=-1$ for $j=1,k$ and $\omega_j=1$ for every other $j$, we get
		\[\varphi(\omega) = \mathbf{u}_0 - 2(\mathbf{x}_1+\mathbf{x}_k) = -2 \mathbf{x}_1.\]
		Since $\varphi$ is a Hilbert point and $\mathbf{x}_1 \neq \mathbf{0}$, we must have $\|\mathbf{x}_1\|=\frac{1}{2}\|\mathbf{u}_0\|$. The same argument shows that $\|\mathbf{x}_2\|=\frac{1}{2} \|\mathbf{u}_0\|$. Choosing $\omega_j=-1$ for $j=1,2,k$ and $\omega_j=1$ for every other $j$, the assumption that $\varphi$ is a Hilbert point implies that $\|\mathbf{x}_1+\mathbf{x}_2\|$ is equal to either $0$ or $\frac{1}{2}\|\mathbf{u}_0\|$. Since $\mathbf{x}_1 \perp \mathbf{x}_2$, we get that
		\[\|\mathbf{x}_1+\mathbf{x}_2\| = \frac{\sqrt{2}}{2}\|\mathbf{u}_0\|,\]
		which is a contradiction. Hence $\mathbf{x}_j = \mathbf{0}$ for every $j\geq3$, and we are in case (a) of Theorem~\ref{thm:rademacher} with only two nonzero vectors in the sequence. 
		
		Second, if
		\[\|\mathbf{u}_0-2(\mathbf{x}_1+\mathbf{x}_2)\|=0,\]
		then Lemma~\ref{lem:geom1}~(a) shows that
		\[\mathbf{x}_1 = \frac{1}{4}\mathbf{u}_0 + \frac{\sqrt{3}}{4} \mathbf{v} \qquad \text{and} \qquad \mathbf{x}_2 = \frac{1}{4}\mathbf{u}_0 - \frac{\sqrt{3}}{4} \mathbf{v}\]
		where $\mathbf{v} \perp \mathbf{u}_0$ and $\|\mathbf{v}\|=\|\mathbf{u}_0\|$. For $j\geq 3$ we have either $\mathbf{x}_j=0$ or $\|\mathbf{u}_0-2\mathbf{x}_j\|=0$ and $\mathbf{x}_j = \frac{1}{2}\mathbf{u}_0$. Since
		\[\mathbf{u}_0-\mathbf{x}_1-\mathbf{x}_2 = \sum_{j=2}^{\infty} \mathbf{x}_j = \frac{1}{2}\mathbf{u}_0,\]
		we see that there is a single index $j\geq3$ such that $\mathbf{x}_j = \frac{1}{2}\mathbf{u}_0$ and $\mathbf{x}_j=\mathbf{0}$ for every other $j\geq3$, and we are in case (c) of Theorem~\ref{thm:rademacher}.
		
		\item[4.] In the case $|J|\geq3$ we rely on Lemma~\ref{lem:geom3}, which says that
		\[\bigg\|\mathbf{u}_0 - 2 \sum_{j \in K} \mathbf{x}_j\bigg\|=\|\mathbf{u}_0\|,\]
		for any finite subset $K$ of $J$. Using Lemma~\ref{lem:geom1} (b) iteratively, we find that $(\mathbf{x}_j)_{j\in J}$ is an orthogonal sequence. It remains to show that there are no other nonzero vectors in $(\mathbf{x}_j)_{j\geq 1}$. Arguing as in the first subcase of the previous case, we conclude that $\mathbf{x}_j = \mathbf{0}$ for every $j$ which is not in $J$. Thus we are in case (a) of Theorem~\ref{thm:rademacher} with $|J|$ nonzero vectors.\qedhere 
	\end{enumerate}
\end{proof}

\bibliographystyle{amsplain} 
\bibliography{hilberthilbert}

\end{document}